\documentclass[11pt]{article}

\usepackage{amsmath, amssymb, amsthm, enumerate}
\usepackage{graphicx, caption, subcaption}
\usepackage{xcolor}
\usepackage{cancel} %strikethrough text via cancel command
\usepackage[pagebackref]{hyperref}

\usepackage{xspace}

\usepackage[margin=1.5in, marginparwidth=1.3in]{geometry}

\usepackage{titlesec}
\theoremstyle{plain}
\newtheorem{thm}{Theorem}[section]
\newtheorem{prop}[thm]{Proposition}
\newtheorem{lemma}[thm]{Lemma}

\theoremstyle{definition}
\newtheorem{defn}[thm]{Definition}

\newtheorem{question}[thm]{Question}

\newtheorem{rmkx}[thm]{Remark}
\newenvironment{rmk}
{\pushQED{\qed}\rmkx}
{\popQED\endrmkx}

\newcommand{\Z}{\mathbb{Z}}
\newcommand{\R}{\mathbb{R}}

\newcommand{\osc}{\operatorname{osc}}

\newcounter{bencomments}

\title{Unbounded bunching of saddle connections on the golden~L }  
\author{Benjamin Dozier \thanks{Department of Mathematics, Cornell University, \href{mailto:benjamin.dozier@cornell.edu}{\nolinkurl{benjamin.dozier@cornell.edu}}. Research supported in part by NSF Grant DMS-2247244 and the Simons Foundation.
}}

\begin{document}
\maketitle

% % ABSTRACT
\begin{abstract}
  We show there is a translation surface (the golden L) that has \emph{unbounded bunching}:  for every positive integer $K$ there exists a ball $B$ of radius $1$ in $\R^2$ that contains at least $K$ vectors that are periods of saddle connections on this surface.  
\end{abstract}

% TABLE OF CONTENTS
%\setcounter{tocdepth}{1}  % set depth shown in table of contents
%\tableofcontents
%\setcounter{tocdepth}{3} % set depth shown in sidebar contents

% -Cite Chenxi paper

% -Cite Barak's student's paper

% -Lelievre conjecture 

\section{Introduction}
\label{sec:intro}
A \emph{translation surface} is a pair $(X,\omega)$, where $X$ is a compact Riemann surface and $\omega$ is a holomorphic $1$-form on $X$.  Equivalently, a translation surface is a finite collection of disjoint polygons in the plane, with each edge paired with a parallel edge of the same length, such that paired edges lie on opposite sides of the polygons they bound.  A \emph{saddle connection} on $(X,\omega)$ is a geodesic segment that starts and ends at a zero of $\omega$, with no zeros on the interior of the segment.  Integrating the form $\omega$ over a saddle connection gives a complex number, its \emph{period}.  It records the length and direction of the saddle connection.  
% , up to a cut-and-paste equivalence relation.  

\begin{defn} 
We say  a translation surface $(X,\omega)$ has \emph{unbounded bunching} if for every positive integer $K$ there exists a ball $B$ of radius $1$ in $\R^2$ that contains at least $K$ vectors that are periods of saddle connections on $(X,\omega)$.  
\end{defn}

\begin{question} \label{q:bunch} 
  Does there exist a translation surface with unbounded bunching?  
\end{question}

\begin{figure}
    \centering
    \includegraphics[width=0.4\linewidth]{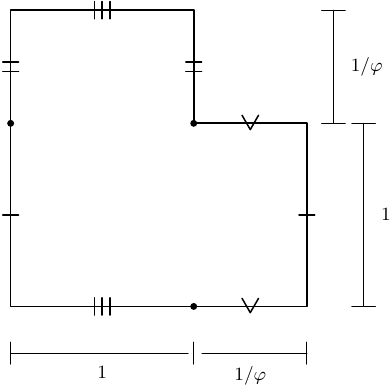}
    \caption{The golden L translation surface, presented as a polygon with edge identifications.  The resulting differential has a single zero of order 2.}
    \label{fig:goldenL}
  \end{figure}
  
Let $\phi := \frac{1+\sqrt{5}}{2}$ be the golden ratio, which satisfies the quadratic equation $\phi^2-\phi-1=0$.
The \emph{golden L} translation surface is the one given in Figure \ref{fig:goldenL}.  

In this note, we answer Question \ref{q:bunch} in the affirmative.  

\begin{thm} \label{thm:goldenLbunch}
  The golden L has unbounded bunching.  
\end{thm}

The proof that we give is elementary, using only basic properties of quadratic irrationals.  It is somewhat intricate since one ends up considering products of powers of 4 matrices.  

\paragraph{Relation to prior work.}
Saddle connections and cylinders on translation surfaces have been of great interest for many decades.  Masur proved coarse upper and lower quadratic bounds on the growth of saddle connections and cylinders of length at most $R$ \cite{masur88, masur90}.  Our initial interest in Question \ref{q:bunch} was that if no surface had unbounded bunching, this would have provided a different explanation for the coarse quadratic upper bound, and would have also implied coarse upper bounds on saddle connections with periods lying in other shapes, such as strips of fixed height and growing width.

Wu proved that on non-arithmetic Veech surfaces one can find arbitrarily small balls that contain at least two saddle periods \cite{wu16}.  However, the argument does not readily generalize to give more than a fixed constant number of saddle connections (related to the maximal number of cylinders in a fixed direction) in balls of fixed radius.

Samuel Leli\`evre conjectured a stronger version of Theorem \ref{thm:goldenLbunch}, namely that for any $K$ and $\epsilon>0$ one can find a \emph{segment} of length $\epsilon$ that contains at least $K$ saddle periods for the golden L.  % see email Jul 8, 2018
Bashan has made progress towards this conjecture, showing the $K=3$ case \cite{bashan24}.  To our knowledge, the general case remains open.

\paragraph{AI use and formalization.}  The proof was found via extensive use of OpenAI's ChatGPT 5.5 Pro.  The exposition presented here is entirely human-written.  A formal statement of the key new result in the proof, Proposition \ref{prop:bunch}, was written in the Lean language and carefully checked (the deduction of our main theorem from this proposition is routine and short, but formalizing this deduction in Lean would involve formalizing many concepts that are not currently in Mathlib).  Using OpenAI's Codex, a formal Lean proof of this proposition was produced, which successfully compiles.  Thus Proposition \ref{prop:bunch} can be considered ``computer verified''.   See the attached ancillary files:
\begin{itemize}
\item \verb+GoldenLUnboundedBunchMathlibStatement.lean+ contains the formalization of the statement of Proposition \ref{prop:bunch},
\item \verb+GoldenLUnboundedBunchMathlib.lean+ contains both the formal statement and proof.  
\end{itemize}

\paragraph{Acknowledgements.}

We thank Joseph Helfer for help with Lean.

\section{Proof of Theorem \ref{thm:goldenLbunch}}
\label{sec:proof}

Define 
\[
  U := \begin{pmatrix} 1 & \phi \\ 0 & 1 \end{pmatrix}, \quad
  L := \begin{pmatrix} 1 & 0 \\ \phi & 1 \end{pmatrix}, \quad
  e_1 := \begin{pmatrix} 1 \\ 0 \end{pmatrix}.
\]

Our main theorem will follow easily once we have established the following proposition concerning the above matrices acting on $\R^2$.

\begin{prop} \label{prop:bunch} 
Let $\Gamma$ be the group $\langle U,L\rangle$.  Then for any positive integer $K$, there exist $\gamma_0,\ldots,\gamma_{K-1}  \in \Gamma$ and a ball $B\subset \R^2$ of radius $1$ such that $\gamma_0 e_1 , \ldots, \gamma_{K-1} e_1$ all lie in $B$ and are distinct.    
%$\gamma_0,\ldots,\gamma_{K-1} \in \Gamma$ and a ball $B\subset \R^2$ of radius $1$, such that $\gamma_i e_1 \in B$ for $i=1,\ldots,K$, and $\gamma_i e_1 \ne \gamma_j e_1$ for $i\ne j$.  
\end{prop}

\begin{rmk}
  In fact the above proposition and the proof we give below can be modified to work for any $\phi>0$ with $\phi^2$ a quadratic irrational.
  %This was verified using ChatGPT, and a formal proof was also produced.  
  
\end{rmk}
Before beginning the proof of the proposition, we set up some notation.  

\begin{itemize}
\item Define $\rho:=\phi^2=\phi+1 = \frac{3+\sqrt 5}{2}$.  Define the Galois conjugate $\rho':= \frac{3-\sqrt 5}{2}$.  Note that $\rho+\rho'=3$ and $\rho\rho'=1$.

\item Define $\alpha_r := 1 + r\rho$.
  
\item 
  Given a family of real numbers $(d_r)_{r\in I}$, define the \emph{oscillation}
\begin{align*}
  \osc_I d_r := \sup_{r,s\in I} |d_r-d_s|.
\end{align*}
  
\end{itemize}

\begin{proof}[Proof of Proposition \ref{prop:bunch}]
  It will suffice to consider group elements of the form
  \begin{align}
    U^m L^nU^r Le_1 &= U^m \begin{pmatrix} \alpha_r\\ \phi(1+n\alpha_r) \end{pmatrix} \label{eq:Um} \\
                    & =\begin{pmatrix} \alpha_r + m\rho(1+n\alpha_r) \\ \phi(1+n\alpha_r) \end{pmatrix}. \label{eq:matrix-prod}
  \end{align}

  The idea is to take $m$ some large fixed number, and then choose $K$ different values of $r$ and $n$.  We can pick out the terms in the $x$-coordinate above that depend on $r$ or $n$; summing them gives $\alpha_r + m\rho n\alpha_r$.  We want to choose $r$ varying in a block of consecutive integers of size $K$, and then choose $n_r$ depending on $r$,  so that this sum is approximately constant.  Lemma \ref{lemma:dioph-approx} below accomplishes this, as we now see.   For the $R,M,A,n_r$ given by that lemma, we get using \eqref{eq:nalpha} that the $x$-coordinate of \eqref{eq:matrix-prod} is
  \begin{align}
    \label{eq:x-coord}
    (1+r\rho) + M\rho(1+A-r/M + E_r/M)  = 1+ M\rho(1+A) +\rho E_r,
  \end{align}
  for all $r\in \{R,R+1,\ldots, R+K-1\} $, and since $\osc_I E_r < \frac{1}{8\rho}$, we get that the oscillation of the above is less than $1/8$. 

  Now we consider the $y$-coordinate of  \eqref{eq:matrix-prod}, which Lemma \ref{lemma:dioph-approx}  has also been crafted to handle.  This coordinate becomes 
  \begin{align}
    \label{eq:y-coord}
    \phi(1+A-r/M + E_r/M),
  \end{align}
  and since $\osc_I (-r/M + E_r/M) < \frac{1}{8\rho} < \frac{1}{8\phi}$, we get that the oscillation of the above is less than $1/8$.  

    Now define $\gamma_i  := U^M L^{n_{R+i}} U^{R+i} L$, for $i=0,\ldots, K-1$.  Take the ball $B$ to be centered at $\gamma_0e_1$ (in fact any $\gamma_ie_1$ would work).  Combining the oscillation estimates for $x,y$ coordinates in  \eqref{eq:x-coord} and \eqref{eq:y-coord} respectively, we get that all the $\gamma_ie_1$ lie in this $B$, as desired.  

    It remains to show that $\gamma_0e_1,\ldots, \gamma_{K-1}e_1$ are distinct.   Suppose the contrary.  Since they correspond to products with the same $m=M$ in the left-hand side of \eqref{eq:Um}, and $U^m$ is invertible, we see from the $x$-coordinate of the right-hand side of that formula that we would have $\alpha_r =\alpha_s$ for some $s\ne r$.  But this is a contradiction, since $\alpha_r = 1+r\rho$.  Thus they are distinct, and we are done.

\end{proof}

\begin{lemma}\label{lemma:dioph-approx}
  Given $K$ a positive integer, there exist integers $R,M,A$ and integers $n_r$ for each $r\in I$, where  $I:=\{R,R+1,\ldots, R+K-1\}$, such that the following holds.  Let $E_r$ be defined to satisfy
  \begin{align}
    n_r\alpha_r = A-\frac r M + \frac{E_r}{M}. \label{eq:nalpha}
  \end{align}
Then
  \begin{align} 
    &\osc_I E_r < \frac{1}{8\rho}, \label{eq:osc-xcoord} \\ 
    &\osc_I \left(-\frac r M + \frac{ E_r}{M}\right) < \frac{1}{8\rho}. \label{eq:osc-ycoord} 
  \end{align}
\end{lemma}

To prove the above lemma, we will divide both sides of \eqref{eq:nalpha} by $\alpha_r$, and then use Kronecker's theorem to produce $A$ and $n_r$.  We first give a useful description of $1/\alpha_r$, a simple exercise in quadratic field arithmetic.    

\begin{lemma}\label{lem:quad-arith} 
  For each integer $r$, we have
  \begin{align*}
    \theta_r:= \frac{1}{\alpha_r} = a_r + b_r \rho,
  \end{align*}
  where $a_r,b_r$ are rational numbers.  Furthermore, $b_r\alpha_r = b_r/\theta_r = \frac{-r}{1+r\rho'}$.  
\end{lemma}

\begin{proof}
  We have
  \begin{align*}
    \frac{1}{\alpha_r} &= \frac{1}{1+r\rho} = \frac{1+r\rho'}{(1+r\rho)(1+r\rho')} =\frac{1+3r-r\rho}{1+3r+r^2}   \\
                       &  = \frac{1+3r}{1+3r+r^2}   + \frac{-r}{1+3r+r^2}   \rho,
  \end{align*}
  so we take $a_r= \frac{1+3r}{1+3r+r^2} $ and $b_r = \frac{-r}{1+3r+r^2} $, which are both rational numbers.  Furthermore,
  \begin{align*}
    b_r\alpha_r =  \frac{-r}{1+3r+r^2} (1+r\rho) = \frac{-r}{(1+r\rho)(1+r\rho')} (1+r\rho) = \frac{-r}{1+r\rho'}.
  \end{align*}
  
\end{proof}

\begin{proof}[Proof of Lemma \ref{lemma:dioph-approx}]
%  [[Maybe pick parameters here: R]
  
  Choose $M>2$ such that $M> 16\rho K$.  Choose $R>2400\rho K^2$.  
  
  Rearranging \eqref{eq:nalpha}, we see that we want $A$ such that 
  \begin{align}\label{eq:thetaA}
    \theta_r A \equiv \frac{\theta_r}{M} (r-E_r) \mod 1
  \end{align}
  for all $r\in I$.  Now pick $Q$ a large integer that is divisible by all the denominators of the $a_r$ and $b_r$ (for $r\in I$).  By Lemma \ref{lem:quad-arith},  $\theta_r A = (a_r+b_r \rho) A$,  and so if we take $A=Qt$, where $t$ is an integer (to be chosen later), we get
    \begin{align}
    \theta_r A = (Qa_r+ Qb_r \rho) t  \equiv Qb_r \rho t \mod 1.  \label{eq:thetarA}
  \end{align}

So substituting this into \eqref{eq:thetaA}, we see that we need to find $t$ that solves the congruences
\begin{align}
  Qb_r \rho t \equiv  \frac{\theta_r}{M} (r-E_r) \mod 1,    \label{eq:qbrt}
\end{align}
for all $r\in I$, for an $E_r$ that has small oscillation over $I$.  The idea is that if the right hand side of the above only depended on $r$ through a factor of $b_r$, we could cancel it from both sides, and then we would only need to solve a single congruence.  %This will not be possible on the nose, but will be true up to some small error.

To this end, we will pick a suitable $C$ and define $\tilde E_r$ by
\begin{align}
  \label{eq:cbr}
  -Cb_r = \theta_r(r-\tilde E_r).
\end{align}
  This will allow us to cancel as suggested.  We will then \emph{approximately} solve \eqref{eq:thetaA} with this $\tilde E_r$, and then our choice of $E_r$ will absorb the small approximation error.   We need to arrange that our $\tilde E_r$ has small oscillation (and then $E_r$ will too). In order to get this, we will pick $C$ such that the derivative $\frac{d}{dr} \tilde E_r$ is small.  To this end, we write $\tilde E_r$ in terms of $C$, giving $\tilde E_r =r+ Cb_r/\theta_r = r - \frac{Cr}{1+r\rho'}$, where the second equality is from Lemma \ref{lem:quad-arith}.  Then we get 
\begin{align}
  \label{eq:deriv}
  \frac{d}{dr} \tilde E_r = 1 - \frac{C}{(1+r\rho')^2}.  
\end{align}
So we take $C:= (1+R\rho')^2$, which makes $\frac{d}{dr} \tilde E_r  $ equal to $0$ at $r=R$, the beginning of our block.  To estimate it on all of $I$, we observe that, with our chosen value of $C$, we have the coarse bound $\left|\frac{d^2}{dr^2} \tilde E_r \right| < 100 R^2/r^3 \le 100/R$. %since our $R$ is large relative to the size $K$ of the block (here, and in the following, the implicit constant in $O(\cdot)$ is absolute).
Integrating, and using \eqref{eq:deriv}, we get, for each $r\in I$, that  $\left|\frac{d}{dr} \tilde E_r \right|  \le 100K/R$.  Then integrating again, we find
\begin{align}
  \label{eq:osc-tilde}
  \osc_I \tilde E_r \le  100  K^2/R  <  \frac{1}{24\rho},
\end{align}
where to get the last inequality we have used our choice of $R$ from the beginning of the proof.  

We now turn to approximately solving \eqref{eq:qbrt}, with $E_r$ replaced by $\tilde E_r$.  Using \eqref{eq:cbr}, we find that we need to solve
\begin{align*}
  Qb_r\rho t \equiv \frac{-Cb_r}{M} \mod 1
\end{align*}
i.e.
\begin{align}
  \label{eq:mod1}
  b_r(Q t \rho + C/M) \equiv 0 \mod 1.
\end{align}
This we can approximately solve using Kronecker's theorem: since $\{t\rho\}_{t\in \Z^+}$ is dense mod $1$, we can find integers $t, \ell$ such that
\begin{align*}
  \left| t \rho + \frac{C}{QM} - \ell \right| < \left| \frac{1}{Qb_r (24\rho M(1+2R\rho))} \right|
\end{align*}
for each $r\in I$ (we are using that $I$ is finite). Then multiplying through by $|Qb_r|$ gives
\begin{align}
  \left| Qb_r t\rho + \frac{Cb_r}{M} - Qb_r \ell\right| < \frac{1}{24\rho M(1+2R\rho)}. %\label{eq:abs-val}
\end{align}
Since $Qb_r \ell$ is an integer, we have solved \eqref{eq:mod1} up to error $1/(24\rho M(1+2R\rho))$, for all $r\in I$.   

Now we unwind everything to complete the proof.  %We have just shown that, taking $n_r=Qb_r\ell$ for $r\in I$, we have
From the previous inequality, and then using \eqref{eq:thetarA}, \eqref{eq:cbr},
we find that there are integers $n'_r$ and $n_r$ such that 
\begin{align*}
  \frac{1}{24\rho M(1+2R\rho)} &> \left| n'_r - \left(Qb_rt\rho  + \frac{Cb_r}{M}\right) \right| \\
  &= \left| n_r - \left(\theta_r A + \frac{-\theta_r(r-\tilde E_r)}{M}\right) \right|. 
\end{align*}
Multiplying by $\alpha_r = 1/\theta_r$ gives
\begin{align*}
  \frac{\alpha_r}{24\rho M(1+2R\rho)} > \left| n_r\alpha_r - \left(A - \frac{r}{M} + \frac{\tilde E_r}{M}\right) \right|. 
\end{align*}
So for these $r, n_r$, we get from \eqref{eq:nalpha}, that
\begin{align}
  \label{eq:erdiff}
  |E_r- \tilde E_r| <\frac{M\alpha_r}{24\rho M(1+2R\rho)} < \frac{1}{24\rho},
\end{align}
where in the last inequality, we have used that $\alpha_r = 1 + r\rho$, which for $r\in I$ is less than  $1+ 2R\rho$, since $R$ is greater than $2K$.

Combining \eqref{eq:erdiff} with \eqref{eq:osc-tilde} gives 
\begin{align*}
  \osc_I E_r \le \osc_I \tilde E_r + 2\sup_{r\in I} |E_r - \tilde E_r| <  \frac{1}{24\rho} + 2\left(\frac{1}{24\rho}\right)  = \frac{1}{8\rho},
\end{align*}
which establishes the desired \eqref{eq:osc-xcoord}.  It is then easy to deduce the desired \eqref{eq:osc-ycoord}  as well: notice that, by our choice of $M$ from the beginning of the proof, we have $\osc_I (-r/M) \le K/M < \frac{1}{16\rho}$ and $\osc _I (E_r/M) \le \osc_I(E_r)/M < \frac{1}{16\rho},$ hence
\begin{align*}
  \osc_I \left(-\frac r M + \frac{ E_r}{M}\right) < \frac{1}{16\rho} + \frac{1}{16\rho} = \frac{1}{8\rho}.
\end{align*}

\end{proof}

\begin{proof}[Proof of Theorem \ref{thm:goldenLbunch}]
It is easy to check that $U,L$ are in the Veech group of the golden L, i.e. the matrices stabilize the surface under the action of $SL_2(\R)$ on the ambient stratum.  And $e_1$ is a period of a saddle connection (and in fact also a cylinder), e.g. the one at bottom left of Figure \ref{fig:goldenL}.  The whole orbit of $e_1$ under the Veech group consists of saddle connection periods.  So our theorem follows from Proposition \ref{prop:bunch}.  
\end{proof}

{\footnotesize
\bibliographystyle{amsalpha}
  \bibliography{sources}%{}

@incollection{masur88,
  author    = {Howard Masur},
  title     = {Lower Bounds for the Number of Saddle Connections and Closed Trajectories of a Quadratic Differential},
  booktitle = {Holomorphic Functions and Moduli I},
  series     = {Mathematical Sciences Research Institute Publications},
  volume      = {10},
  pages      = {215--228},
  publisher  = {Springer},
  address    = {New York},
  year       = {1988},
  doi        = {10.1007/978-1-4613-9602-4_20}
}

@article{masur90,
  author  = {Howard Masur},
  title   = {The Growth Rate of Trajectories of a Quadratic Differential},
  journal = {Ergodic Theory and Dynamical Systems},
  volume  = {10},
  number  = {1},
  pages   = {151--176},
  year    = {1990},
  doi     = {10.1017/S0143385700005459}
}

@article {wu16,
    AUTHOR = {Wu, Chenxi},
     TITLE = {Delon\'e{} property of the holonomy vectors of translation
              surfaces},
   JOURNAL = {Israel J. Math.},
  FJOURNAL = {Israel Journal of Mathematics},
    VOLUME = {214},
      YEAR = {2016},
    NUMBER = {2},
     PAGES = {733--740},
      ISSN = {0021-2172,1565-8511},
   MRCLASS = {37D40 (37A45)},
  MRNUMBER = {3544700},
MRREVIEWER = {Thomas\ Ward},
       DOI = {10.1007/s11856-016-1357-y},
       URL = {https://doi.org/10.1007/s11856-016-1357-y},
}

@ARTICLE{bashan24,
    AUTHOR = {Bashan, Sahar},
     TITLE = {On nonuniformly discrete orbits},
   JOURNAL = {Comb. Number Theory},
  FJOURNAL = {Combinatorics and Number Theory},
    VOLUME = 14,
      YEAR = 2025,
    NUMBER = {3-4},
     PAGES = {271--280},
      ISSN = {2996-2196,2996-220X},
   MRCLASS = {37A17 (06B25 11J70)},
  MRNUMBER = 4989187,
       DOI = {10.2140/cnt.2025.14.271},
       URL = {https://doi.org/10.2140/cnt.2025.14.271},
}
}

\end{document}